\documentclass[11pt]{amsart}
\usepackage{fullpage}
\usepackage{graphicx}
\usepackage{amsfonts}
\usepackage{amsthm}
\usepackage{amsmath}
\usepackage{amssymb}
\usepackage{comment}
\usepackage[arrow,matrix,curve,color]{xy}
\usepackage{pb-diagram,pb-xy}
\usepackage{hyperref}   
\usepackage{enumitem}
\usepackage[normalem]{ulem}
\usepackage{color}
\usepackage[normalem]{ulem}
\usepackage{tikz}
\usepackage{xspace}
\definecolor{light-blue}{rgb}{0.8,0.85,1}
\definecolor{light-red}{rgb}{1,.4,.4}
\definecolor{purp}{rgb}{.7,.3,1}
\definecolor{yel}{rgb}{1,1,.5}
\definecolor{cy}{rgb}{0,1,1}
\usepackage{colortbl}
\usepackage{multirow}
\usepackage{array}

\theoremstyle{plain}
\newtheorem{theorem}{Theorem}[section]
\newtheorem{corollary}[theorem]{Corollary}
\newtheorem{lemma}[theorem]{Lemma}
\newtheorem{proposition}[theorem]{Proposition}

\theoremstyle{definition}
\newtheorem{remark}[theorem]{Remark}
\newtheorem{definition}[theorem]{Definition}
\newtheorem{example}[theorem]{Example}
\newtheorem{conjecture}[theorem]{Conjecture}
\newtheorem{question}[theorem]{Question}

\newcommand{\Vol}{\mathrm{Vol}}

\newcommand{\sL}{\mathsf{L}}

\newcommand{\co}{\colon\,}

\newcommand{\bR}{\mathbb R}
\newcommand{\bC}{\mathbb C}

\newcommand{\bZ}{\mathbb Z}

\newcommand{\bP}{\mathbb P}

\newcommand{\cA}{\mathcal A}

\newcommand{\cC}{\mathcal C}

\newcommand{\cH}{\mathcal H}

\newcommand{\cK}{\mathcal K}

\newcommand{\cO}{\mathcal O}

\newcommand{\cR}{\mathcal R}

\newcommand{\pt}{\text{\textup{pt}}}
\newcommand{\lp}{\textup{(}}
\newcommand{\rp}{\textup{)}}

\newcommand{\ind}{\operatorname{ind}}

\newcommand{\Tr}{\operatorname{Tr}}

\newcommand{\inddiff}{\operatorname{inddiff}}

\newcommand{\comm}[1]{}
\newcommand{\Rtw}{R^{\text{\textup{tw}}}}
\newcommand{\Rgen}{R^{\text{\textup{gen}}}}
\newcommand{\vol}{\text{\textup{vol}}}
\newcommand{\gen}{{\mathrm{gen}+}}
\newcommand{\tw}{{\mathrm{tw}+}}

\numberwithin{equation}{section}
\makeatletter
\@namedef{subjclassname@2020}{%
  \textup{2020} Mathematics Subject Classification}
\makeatother

\title[Spin$^c$ Metrics]{Generalized positive scalar curvature on
  spin$^c$ manifolds}
\author{Boris Botvinnik}
\address{Department of Mathematics\\
University of Oregon\\
Eugene OR 97403-1222, USA} 
\email[Boris Botvinnik]{botvinn@uoregon.edu}
\urladdr{http://pages.uoregon.edu/botvinn/}
\author{Jonathan Rosenberg}
\address{Department of Mathematics\\
University of Maryland\\
College Park, MD 20742-4015, USA} 
\email[Jonathan Rosenberg]{jmr@umd.edu}
\urladdr{http://www2.math.umd.edu/\raisebox{-3pt}{~}jmr/}
  
\begin{document}
\begin{abstract}
  Let $(M,L)$ be a {(compact)}
  non-spin spin$^c$ manifold.  Fix a Riemannian metric
  $g$ on $M$ and a connection $A$ on $L$, and let $D_L$ be the associated
  spin$^c$ Dirac operator. Let $\Rtw_{(g,A)}:=R_g + 2ic(\Omega)$ be the
  \emph{twisted scalar curvature} (which takes values in the
  endomorphisms of the spinor bundle), where $R_g$ is the scalar
  curvature of $g$ and $2ic(\Omega)$ comes from the curvature
  $2$-form $\Omega$ of the
  connection $A$. Then the Lichnerowicz-Schr\"odinger formula for the
  square of the Dirac operator takes the form $D_L^2 =\nabla^*\nabla+
  \frac{1}{4}\Rtw_{(g,A)}$. In a previous work we proved that a closed non-spin
  simply-connected spin$^c$-manifold $(M,L)$ of dimension $n\geq 5$
  admits a pair $(g,A)$ such that $\Rtw_{(g,A)}>0$ if and only if the
  index $\alpha^c(M,L):=\ind D_L$ vanishes in $K_n$. In this paper we
  introduce a scalar-valued \emph{generalized scalar curvature}
  {$\Rgen_{(g,A)}:=R_g - 2|\Omega|_{op}$}, where $|\Omega|_{op}$ is the
  pointwise operator norm of Clifford multiplication $c(\Omega)$,
  acting on spinors. We show that the positivity condition on the
  operator $\Rtw_{(g,A)}$ is equivalent to the positivity of the scalar
  function $\Rgen_{(g,A)}$. We prove a corresponding trichotomy theorem
  concerning the curvature $\Rgen_{(g,A)}$, and study its
  implications. We also show that the space $\cR^{\gen}(M,L)$ of pairs
  $(g,A)$ with $\Rgen_{(g,A)}>0$ has non-trivial topology, and address a
  conjecture about non-triviality of the ``index difference'' map.
\end{abstract}
\thanks{BB was partially supported by Simons collaboration grant 708183}
\maketitle

\section{Introduction}
\label{sec:intro}
\subsection{Motivation and results}
Recall that a spin$^c$ manifold $(M,L)$ is a smooth oriented manifold
$M$, $\dim M=n $, equipped with a choice of a complex line bundle $L$
with the property that $c_1(L)$ reduces mod $2$ to $w_2(M)$.  Then $M$
admits a spin$^c$ Dirac operator $D_L$, which can be viewed as the
spin Dirac operator $D$ on $M$ (which always exists locally, but won't
exist globally on $M$ if $w_2(M)\ne 0$) twisted by the ``line bundle''
$L^{1/2}$.  (Note that $L^{1/2}$ is also not globally well-defined,
but the indeterminacies in $D$ and in $L^{1/2}$ cancel each other
out.)  To define the Dirac operator
$D_L$, one needs some extra data: a Riemannian
metric $g$ on $M$, a hermitian bundle metric $h$ on $L$, and a
connection $A$ on $L$.  The $K$-homology class $[D_L]\in K_n(M)$ of
the operator does not depend on any of these choices, nor does the
index $\alpha^c(M,L):=\ind D_L=p_*([D_L])\in K_n$, where $p\co
M\to\pt$.  The Lichnerowicz-Schr\"odinger formula for the square of
the Dirac operator takes the form
\begin{equation}\label{Lich-formula}
  D_L^2 =\nabla^*\nabla+ \frac{R_g}{4} + \cR_L,
\end{equation}
where $R_g$ is the scalar curvature of $M$ and
$\cR_L=\frac{1}{2}ic(\Omega)$ is $\frac{i}{2}$ times Clifford
multiplication by the curvature $2$-form $\Omega$ of the connection $A$
acting on spinors; see \cite[Theorem D.12]{lawson89:_spin}.

We define the \emph{twisted scalar curvature}
$\Rtw_{(g,A)}$ to be the function $\Rtw_{(g,A)}:=R_g + 2ic(\Omega)$,
taking values in sections of the endomorphism bundle of the spinor
bundle, so that $D_L^2 =\nabla^*\nabla+ \frac{1}{4}\Rtw_{(g,A)}$.
This function $\Rtw_{(g,A)}$ (except for a factor of $4$)
appeared in \cite{MR1162671}, where it was called the
\emph{curvature endomorphism} $\cK$ for the spin$^c$ spinor bundle. So
when $n$ is even, the condition that the operator $\Rtw_{(g,A)}$ is
positive-definite forces $\alpha^c(M,L)=\ind D_L$ to vanish. We use
the notation $\Rtw_{(g,A)}>0$ and call this condition \emph{twisted
positive scalar curvature} or \emph{twisted psc} for short.

Note that
the condition $\alpha^c(M,L)=0$ depends very much on the Chern class
$c_1(L)$.  For example, when $M=\bC\bP^2$, which is spin$^c$ but not
spin, $c_1(L)\in H^2(\bC\bP^2;\bZ)\cong \bZ$ must be odd since it must
reduce mod $2$ to $w_2(M)\ne 0$, and the condition $\alpha^c(M,L)=0$
is satisfied when $c_1(L)=\pm 1$ but not when $|c_1(L)|\ge 3$.  (See
for example \cite[Theorem 3.13]{MR508087}.)

In our previous paper \cite{BJ}, we proved a surgery theorem
for the twisted psc condition:
\begin{theorem}[{\cite[Theorem 4.2]{BJ}}]
\label{them:spincsurgery}
Let $(M,L)$ be a closed spin$^c$ manifold.  Assume that $M$ admits a
Riemannian metric $g$ and the line bundle $L$ admits a hermitian
bundle metric $h$ and a unitary connection $A$ such that
$\Rtw_{(g,A)}>0$. Let $(M', L')$ be obtained from $(M, L)$
by spin$^c$ surgery in codimension $\ge 3$.  Then $(M', L')$
admits a hermitian bundle metric $h'$, a Riemannian metric $g'$, and a
unitary connection $A'$ such that $\Rtw_{(g',A')}>0$.
\end{theorem}
Then from this and a bordism argument,
we proved the following:
\begin{theorem}[see {\cite[Theorem 4.2]{BJ} and \cite[Corollary 5.2]{BJ}}]
\label{them:spincpsccondition}  
Let $(M,L)$ be a simply connected closed spin$^c$ manifold
which is not spin, with $n=\dim M\geq 5$. Then $(M,L)$ admits a
Riemannian metric $g$, a
hermitian bundle metric $h$, and a unitary connection $A$ such that
$\Rtw_{(g,A)}>0$, if and only if the index $\alpha^c(M,L)$ vanishes in
$K_n$.
\end{theorem}
Given a spin$^c$ manifold $(M,L)$ we define the
\emph{generalized scalar curvature} $\Rgen_{(g,A)}$ of a triple
$(g,h,A)$, where $g$ is a Riemannian metric on $M$, $A$ is a
connection on the line bundle $L$, and $h$ is a metric on the bundle
$L$, to be the \emph{scalar-valued} function $\Rgen_{(g,A)}:=R_g -
2|\Omega|_{op}$, where $|\Omega|_{op}$ is the pointwise operator
norm of Clifford multiplication $c(\Omega)$, acting on
spinors.\footnote{Recall that if $x\in M$,
$|\Omega|_{op}(x)=\max \{\vert c(\Omega)\psi(x)\vert : \vert
\psi(x)\vert =1\}$, where the maximum is taken over all spinors
$\psi$ with norm $1$ at the point $x$.}
{Note that when $M$ is spin, one can choose $L$ to be the trivial
  bundle and the connection $A$ to be flat, so in  this case
  $\Rgen_{(g,A)}=R_g$.  So in that sense $\Rgen_{(g,A)}$ generalizes the
  notion of scalar curvature when we pass from spin manifolds
  to spin$^c$ manifolds.}

We observe that the positivity condition on the operator $\Rtw_{(g,A)}$
is equivalent to the positivity of the scalar function $\Rgen_{(g,A)}$ (see
Lemma \ref{lem:gpsccond}).

Recall the trichotomy theorem of Kazdan-Warner, which can be
formulated as follows:
\begin{theorem}[Kazdan-Warner {\cite{MR365409,MR0394505,MR0375154}}]
\label{thm:KW-int}
  Every closed connected manifold $M$ of dimension $\ge 3$
  falls into exactly one of the following three classes:
  \begin{enumerate}
  \item those admitting a psc metric, in which case every smooth function
    on $M$ is the scalar curvature of some metric;
  \item those admitting a metric $g$ with $R_g\equiv 0$ but not a metric with
    $R_g\ge 0$ and $R_g$ positive somewhere --- in this case a metric with
    $R_g\equiv 0$ is necessarily Ricci-flat, and the possible scalar
    curvature functions on $M$ are $0$ and the functions negative somewhere;
  \item all other manifolds, those not admitting any metric with
    nonnegative scalar curvature --- in this case, the possible scalar
    curvature functions of metrics on $M$ are exactly those functions
    which are negative somewhere.
  \end{enumerate}
\end{theorem}
We prove a corresponding trichotomy theorem concerning the generalized scalar
curvature $\Rgen_{(g,A)}$:
\begin{theorem}
\label{thm:trichotomy-int}
  Every closed connected spin$^c$ manifold $(M,L)$ of dimension at
  least three falls into exactly one of the following three classes:
  \begin{enumerate}
  \item those admitting a pair $(g,A)$ with $\Rgen_{(g,A)}>0$, in which
    case every smooth function on $M$ is the generalized scalar
    curvature of some pair $(g',A')$, $g'$ a Riemannian metric on $M$ and
    $A'$ a connection on $L$;
  \item those admitting a pair $(g,A)$ with $\Rgen_{(g,A)}\equiv 0$ but not
    a pair $(g,A)$ with $\Rgen_{(g,A)}$ positive everywhere --- in this
    case a smooth function on $M$ is
    the generalized scalar curvature of some pair $(g',A')$ 
    if and only if it is either $\equiv 0$ or negative somewhere on $M$;
  \item all other manifolds, those not admitting any $(g,A)$ with
    $\Rgen_{(g,A)}\ge 0$ --- in this case, a smooth function on $M$ 
    is the generalized scalar curvature of some pair $(g',A')$ if and only if
    it is negative somewhere on $M$.
  \end{enumerate}
\end{theorem}
Recall that the conformal Laplacian
$\sL_{g}=-\frac{4(n-1)}{n-2}\Delta_g + R_g$ and its lowest
eigenvalue $\mu_1(\sL_{g})$ played a major role in proving Theorem
\ref{thm:KW-int}, where the sign of $\mu_1(\sL_{g})$ (which is 
invariant under a conformal change of the metric $g$) points to the
corresponding class (1), (2) or (3).\footnote{We use the same sign
  convention as Kazdan and Warner, that $\Delta_g$ is a nonpositive
  operator.  This explains the initial minus sign in the formula
for $\sL_{g}$.} Similarly, the generalized
scalar curvature $\Rgen_{(g,A)}$ is closely related to the principal
eigenvalue $\mu_1(\sL_{(g,A)})$ of the \emph{generalized conformal
Laplacian}
\begin{equation*}
{\sL_{(g,A)}}=-\frac{4(n-1)}{n-2}\Delta_g +
{\Rgen_{(g,A)}}=-\frac{4(n-1)}{n-2}\Delta_g + R_g -
2|\Omega|_{op}
= \sL_g - 2|\Omega|_{op}.
\end{equation*}  
Here again, the sign of $\mu_1(\sL_{(g,A)})$ is invariant under a
conformal change of the metric $g$. The condition $\mu_1(\sL_{(g,A)})>0$
corresponds to the case (1): then there exists a pair $(g',A)$ (where
$g'$ is conformal to $g$) such that $\Rgen_{(g',A)}>0$. In turn, this
means that $\alpha^c(M,L)=0$ as we have seen above. Thus if
$\alpha^c(M,L)\ne 0$, there are two possibilities: either there is a
pair $(g,A)$ with $\mu_1(\sL_{(g,A)})=0$, or else $\mu_1(\sL_{(g,A)})<0$ for any
pair $(g,A)$.

The case when $\mu_1(\sL_{(g,A)})=0$ (corresponding to the case (2)) is
especially interesting: just as in the classic case, we prove a
rigidity result as follows:
\begin{theorem}
\label{thm:specialhol-int}
Let $(M,L)$ be a closed connected spin$^c$ manifold which is non-spin, and
assume that $\alpha^c(M,L)\ne 0$. Suppose there are a metric $g$ on $M$ and a
connection $A$ on $L$ for which $\mu_1(\sL_{(g,A)})=0$. Then $(M,g)$ is
pointwise conformal to another metric $g_1$ for which
$\Rgen_{(g_1,A)}$ vanishes identically. Furthermore, $(M,L)$ with metric
$g_1$ admits a parallel spinor, and $(M,g_1)$ is a Riemannian product
of a K\"ahler manifold and a simply connected spin manifold with a
parallel spinor. In particular,
if $M$ does not split as a product, then $M$ is conformally K\"ahler,
and $L$ is either the canonical or the anti-canonical line bundle on $M$.
\end{theorem}
We also observe that there is a partial converse to Theorem
\ref{thm:specialhol-int} that applies to K\"ahler-Einstein Fano varieties.
A very similar formulation appears in \cite{MR1876863}.
\begin{theorem}{\rm ({cf.\ }Goette-Semmelmann,
    {\cite[Theorem 1.9]{MR1876863}})}
\label{thm:specialhol1-int}
Let $(M,g)$ be a closed connected K\"ahler-Einstein manifold with positive
Ricci curvature.  Let $L$ be the anti-canonical line bundle on $M$.
Then $(M,L)$ does not admit a connection $A$ on $L$ with
$\Rgen_{(g,A)}>0$, but with the connection $A$ on $L$ with harmonic
curvature {\lp}necessarily a multiple of $g${\rp}, $\Rgen_{(g,A)}\equiv 0$.
\end{theorem} 
Next, we consider the space $\cR^{\gen}(M,L)$ of triples $(g,h,A)$ such
that $\Rgen_{(g,A)}>0$. Note that the space $\cR^{\gen}(M,L)$ coincides
with the corresponding space $\cR^{\tw}(M,L)$ of triples $(g,h,A)$ such
that $\Rtw_{(g,A)}>0$.

In fact there is no loss of generality in fixing $h$ once and for all
(see Proposition \ref{prop:dontneedh} below), but it would seem at
first that it is important to consider variations not just in the
Riemannian metric $g$, which determines the scalar curvature term
$R_g$, but also in the connection $A$, which determines the extra
curvature term $\cR_L$.  We will show, however, that this is not
really necessary, and that once one has $(g,h,A)\in \cR^{\gen}(M,L)$,
one can change the connection $A$ to a more canonical choice, only
depending on the metric $g$, for which the positivity condition is
still satisfied.

To study the topology of $\cR^{\gen}(M,L)$,
we will use a version of the ``index difference'' homomorphism,
\begin{equation*}
\inddiff(M)\co \pi_q(\cR^+(M))\to KO_{q+n+1}
\end{equation*}
(where $M$ is a closed spin manifold, $\dim M=n\geq 5$, and $\cR^+(M)$ is the
space of psc metrics on $M$) studied in
\cite{MR3683115,MR3681394,MR3956897}. Loosely, it follows from
\cite{MR3681394} that the index difference homomorphism is non-trivial
if the target group $KO_{q+n+1}$ is non-zero.

Similarly, if $(M,L)$ is a closed spin$^c$ manifold with a fixed choice of
basepoint $(g_0,A_0)\in \cR^{\gen}(M,L)$, we define the spin$^c$ index
difference homomorphism
\begin{equation}\label{ind-diff-int}
\inddiff^c(M,L)\co \pi_q(\cR^{\gen}(M,L))\to KU_{q+n+1}\,.
\end{equation}
We would like to address the following
\begin{conjecture}\label{BR-conj}
Let $(M,L)$ be a closed spin$^c$ non-spin manifold
{with $\dim M\geq 5$, and with} 
$\cR^{\gen}(M,L)\neq \emptyset$.  Then the index difference
homomorphism (\ref{ind-diff-int}) is non-trivial if the target group
is non-zero.
\end{conjecture}  
\begin{remark}
In fact there is an index-difference map
$\cR^{\gen}(M,L)\to \Omega^{\infty+n}\mathbf{KU}$ (we use the notation from
\cite{MR3683115}). It is very likely that the techniques from
\cite{MR3683115,MR3681394,MR3956897} could be (and will be)
generalized to the spin$^c$ case.
\end{remark}
In this paper, we prove a couple of easy results providing evidence
for Conjecture \ref{BR-conj} to be true:
\begin{theorem}\label{thm:psc-space}$\mbox{ \ }$
  \begin{enumerate}
\item[{\rm (i)}] Let $(M,L)$ be a closed connected $n$-dimensional
  spin$^c$ manifold, with $n=4k-1\geq 7$, and $\cR^{\gen}(M,L)\neq
\emptyset$. Then $\inddiff^c(M,L)\co
  \pi_0(\cR^{\gen}(M,L))\to \bZ$ is nonzero, and in {fact},
  $\cR^{\gen}(M,L)$ has infinitely many connected components.
\item[{\rm (ii)}] There exists a closed simply connected spin$^c$ manifold
  $(M,L)$, not spin, for which infinitely many homotopy groups of
  $\cR^{\gen}(M,L)$ are non-zero.  
  \end{enumerate}  
\end{theorem}  
\subsection{Plan of the paper} In Section
\ref{sec:more} we study the question: does there always exist a
spin$^c$ line bundle $L$ with $\alpha^c(M,L)=0$? In Section
\ref{sec:genpsc} we analyze the relationship between the curvatures
$\Rgen$ and $\Rtw$ as well as the behavior of
$\cR^{\gen}(M,L)$ under conformal deformations. We prove the trichotomy
theorem in Section \ref{sec:trichotomy} and analyze the rigidity case
in Section \ref{sec:rigidity}. We define and study the index-difference map 
$\inddiff^c(M,L)$ in Section \ref{sec:inddiff}.
{\subsection{Acknowledgment} We would like to thank the referee
  for a careful reading of the paper and for many useful suggestions.}

\section{More About Spin$^c$ Manifolds}
\label{sec:more}
In this section we want to go into some more detail about spin$^c$
manifolds.  First of all, we should clarify some terminology.  Often
the words ``spin$^c$ manifold'' are used to describe a manifold
admitting a spin$^c$ structure, but really one should avoid this since
spin$^c$ structures are rarely unique, and in fact \emph{never are} in
the case we are primarily interested in here, namely simply connected
manifolds $M$ that are not spin but admit a spin$^c$ structure.
That's because $M$ being simply connected implies $H^2(M,\bZ)$ is
torsion free, and $M$ being non-spin but admitting a spin$^c$
structure means $w_2(M)\ne 0$ and is the reduction of an integral
class in $H^2(M,\bZ)$.  Thus there must be an infinite set of
cohomology classes, all representing distinct spin$^c$ structures,
that reduce mod $2$ to $w_2(M)$.  Each such cohomology class is
the first Chern class $c_1(L)$ of a line bundle $L$,
and it is best to refer to the {\emph{pair}} $(M,L)$ as being a spin$^c$
manifold.  (Of course $L$ is only determined up to isomorphism;
isomorphism classes of line bundles, with the operation of tensor
product of bundles, give the Picard group $\operatorname{Pic}(M)$,
isomorphic to $H^2(M,\bZ)$ via the first Chern class $c_1$.)

Suppose $M^n$ is closed, simply connected, and non-spin, but admits a spin$^c$
structure $(M,L)$.  Then the index $\alpha^c(M,L)$ of the spin$^c$
Dirac operator on $(M,L)$ is well-defined.  Note that the index
vanishes identically if $n$ is odd.  Then the following question
naturally arises:
\begin{question}
\label{question:vanishingindex}
Let $M$ be a closed simply connected manifold with even dimension
$n=2k$.  Fix an orientation on $M$. Assume that $w_2(M)\ne 0$ {\lp}so
that $M$ is not spin{\rp} but that $w_2(M)$ is the reduction of an
integral class {\lp}so that $M$ admits spin$^c$ structures
$(M,L)${\rp}.  Does there always exist a spin$^c$ line bundle $L$ with
$\alpha^c(M,L)=0$?  If so, how many such line bundles $L$ satisfy this
condition?
\end{question}
The answer to this is the following:
\begin{theorem}
\label{thm:vanishingindex}
Let $M$ be a simply connected closed manifold with even dimension $n=2k\ge6$.
Fix an orientation on $M$. Assume that $w_2(M)\ne 0$ but that
$w_2(M)$ is the reduction of an integral class.
{If the second Betti number $\beta_2(M)$ is $1$ and there is
a bundle $L$ with $\alpha^c(M,L)\ne 0$, then the
number of bundles $L'$ {\lp}up to isomorphism{\rp} with $\alpha^c(M,L')=0$
is finite.  Furthermore, there exist such manifolds $M$ with
$\alpha^c(M,L)\ne 0$ for all spin$^c$ line bundles $L$.
On other hand, there also exist such manifolds with $\alpha^c(M,L)=0$
for all spin$^c$ line bundles $L$.}
\end{theorem}
\begin{proof}
As we saw above, $H^2(M,\bZ)$ is torsion-free and infinite, and the
even classes in this group (those reducing mod $2$ to $0$) act
simply transitively on the classes of spin$^c$ line bundles compatible
with the orientation. {Let $x=c_1(L)\in H^2(M,\bZ)$. By
\cite[Theorem D.15]{lawson89:_spin},}
\[
\begin{aligned}
\alpha^c(M,L)&=\left\langle \widehat {\mathbf A}(M)e^{x/2},[M]\right\rangle
\\& = \left\langle \left(
\frac{1}{k!}\left(\frac{x}{2}\right)^k + \cdots + 1\right)
\left(1 - \frac{1}{24}p_1 +\cdots\right),[M]\right\rangle\\
&= \frac{1}{2^k\,k!}\langle x^k,[M]\rangle + \text{lower order terms in }x,
\end{aligned}
\]
{and so vanishing of $\alpha^c(M,L)$
is a polynomial condition on $x=c_1(L)$, which we can think of as
living in the real vector space $H^2(M,\bR)$.  If $\beta_2(M)=1$,
this vector space is one-dimensional, and a polynomial equation in
one variable can have only finitely many roots unless it vanishes
identically.  So this proves the first statement.}

{To get an example where $\alpha^c(M,L)=0$ for all $L$,
consider $M=\bC\bP^2\times S^4$. If $x_2$ is the standard generator
of $H^2(\bC\bP^2, \bZ)$ and $y_4$ is the standard generator of
$H^4(S^4, \bZ)$, then $M$ has the cohomology ring
$\bZ[x_2, y_4]/(x_2^3, y_4^2)$ and every class in  $H^2(M,\bZ)$
is of the form $kx_2$.  Furthermore,
$e^{kx_2/2}=1+\frac{kx_2}{2}+\frac{k^2x_2^2}{8}$, while
$\widehat {\mathbf A}(M) = 1 - \frac{x_2^2}{8}$.  Since neither of
these polynomials involves $y_4$, their product must evaluate to
$0$ on $[M]$.}

{Now we give an example where $M$ doesn't admit any spin$^c$ line
bundle $L$ with vanishing index. We take $M=\bC\bP^4\# N$, where $N$
is a $3$-connected spin $8$-manifold with $\widehat A(N)=1$.
(One can construct such an $N$ via the
$E_8$-plumbing; see \cite{Bott-manifold}.) Since $N$ is $3$-connected,
any spin$^c$ line bundle on $M$ comes uniquely from one on $\bC\bP^4$.
But for a line bundle $L$ on $\bC\bP^4$ with Chern class
$c_1(L)=(2k+1)x$, where $x$ is the standard
generator of $H^2(\bC\bP^4,\bZ)$ --- we write things this way since
the Chern class must be odd by the spin$^c$ condition --- we have
\begin{equation*}
\begin{aligned}
\alpha^c(\bC\bP^4,L)&=\langle \widehat{\mathbf A}(\bC\bP^4)e^{(2k+1)x/2},
      [\bC\bP^4]\rangle\\
      &= \text{coefficient of }x^4\text{ in }
      \left(\frac{x/2}{\sinh(x/2)}\right)^5 e^{(2k+1)x/2}\\
      &= \text{coefficient of }x^4\text{ in }
      \left(1-\frac{5x^2}{24} +\frac{3x^4}{128}\right)\times\\
      &\qquad\qquad
      \left(1+\frac{(2k+1)x}{2}+\frac{(2k+1)^2x^2}{8}+\frac{(2k+1)^3x^3}{48}+
      \frac{(2k+1)^4x^4}{384}\right)\\
      &= \frac{1}{24}(k-1)k(k+1)(k+2).
\end{aligned}
\end{equation*}
(This calculation, by the way, is {also} in
\cite[Proposition 4.3]{MR508087} and in \cite[Remark 3.2]{MR1876863}.)
For $k$ an integer, this is always non-negative: it vanishes for
$k=-2, -1, 0, 1$.  We obtain
\begin{equation*}
  \alpha^c(M,L)=\alpha^c(\bC\bP^4,L)+\widehat A(N)\ge 1, 
\end{equation*}
which never vanishes.}
\end{proof}

\begin{remark}
If $\beta_2(M)>1$ and $\alpha^c(M,L)$ doesn't vanish identically,
then it can still happen that
the number of line bundles $L$ with $\alpha^c(M,L)=0$
can be infinite.  An example is $M=\bC\bP^2\# \overline {\bC\bP^2}$,
$\bC\bP^2$ with one point blown up.  We can identify $H^2(M,\bZ)$
with $\bZ x \oplus \bZ y$, where $x^2=-y^2$ is dual to $[M]$ and $xy=0$.
Then if $c_1(L)=ax+by$, we have
\[
\begin{aligned}
  \alpha^c(M,L)&=\left\langle \widehat {\mathbf A}(M)e^{c_1(L)/2},
        [M]\right\rangle\\
  &=\left\langle \left(1-\frac{x^2}{8}+\frac{x^2}{8}\right)e^{ax/2}e^{by/2},
  [M]\right\rangle\\
  &= \frac{a^2-b^2}{8}.
\end{aligned}
\]
(Note that $a$ and $b$ both have to be odd to get a spin$^c$ structure.)
This vanishes whenever $a=\pm b$, so there are infinitely many choices
with vanishing $\alpha^c$.
\end{remark}

\section{Further Aspects of the ``Generalized psc'' Condition}
\label{sec:genpsc}
{Let $(M,L)$ be a spin$^c$ manifold, $h$ and $A$ a hermitian metric and
  unitary connection on the line bundle $L$, as usual.}
First we want to justify our previous assertion that there is no loss
of generality in fixing the metric $h$ on $L$ once and for
all. \footnote{We use this and hereafter omit $h$ from most
  notations.}  Before we do this, we need to fix some terminology.
If {$\omega$} is a (real) $2$-form on a manifold $M$, we denote by
{$|\omega|_{op}$} the pointwise operator norm of Clifford multiplication
{$c(\omega)$}, acting on spinors.  In other words, if $x\in M$,
{%
\begin{equation*}
    |\omega|_{op}(x)=\max \{\vert c(\omega)\psi(x)\vert:\vert\psi(x)\vert =1\},
\end{equation*}
where} the maximum is taken over all spinors $\psi$
with norm $1$ at $x$.
\begin{lemma}
\label{lem:gpsccond}
Let $(M,L)$ be a spin$^c$ manifold. Let $(g,h,A)$ be
  a triple {\lp}$g$ a Riemannian metric on $M$, $h$ a hermitian
  metric on $L$, and $A$ a connection on $L${\rp}. Then the condition
  $\Rtw_{(g,A)}>0$ is equivalent to the condition $\Rgen_{(g,A)}>0$.
\end{lemma}
\begin{proof}
Clifford multiplication $c(\Omega)$ by $\Omega$ is skew-adjoint, and
its eigenvalues come in pairs $\pm i\lambda$, with $\lambda$ real,
so the condition that $\Rtw_{(g,A)}=R_g + 2\cR_L>0$
reduces to saying that $R_g > 2|\Omega|_{op}$,
where $|\Omega|_{op}$ is the pointwise operator
norm of $c(\Omega)$, or that the eigenvalue $i\lambda$ of
$c(\Omega)$ at any point $x\in M$, with $\lambda\ge 0$ as big as
possible, is less than $\frac12 R_g(x)$. {This is the condition
  $\Rgen_{(g,A)}>0$.}

In the other direction, if $\Rgen_{(g,A)}>0$, then the minimal
  eigenvalue of $ic(\Omega)$ is everywhere less than $\frac12 R_g(x)$
  in absolute value, which shows that the self-adjoint operator
  $\Rtw_{(g,A)}$ is positive.
\end{proof}
Thus Lemma \ref{lem:gpsccond} says that
$\Rtw_{(g,A)}>0\Leftrightarrow \Rgen_{(g,A)}>0$, with the latter a
more manageable condition to work with since $\Rgen_{(g,A)}$ is
scalar-valued. {Note that both conditions imply vanishing of
$\alpha^c(M,L)$, because of equation \eqref{Lich-formula}.}
We denote by $\cR(M)$ the space of Riemannian metrics
on $M$, by $\cA(M,L)$ the space of connections on $L$ and by $\cH(L)$
the space of Hermitian metrics on $L$.  Let 
\begin{equation*}
  \cR^{\tw}(M,L):=\bigl\{(g,A,h)\in\cR(M)\times \cA(M,L)\times \cH(L)
  \mid \Rtw_{(g,A,h)}>0\bigr\}.
\end{equation*}
Let $h_0\in \cH(L)$ be a fixed Hermitian
metric on $L$, and let
$\cR^{\tw}(M,L)_{h_0}=\{(g,A,h_0)\in \cR^{\tw}(M,L)\}$.
Since the contractible group $C^{\infty}_+(M)$ of positive smooth functions
acts freely and transitively on $\cH(L)$ by pointwise multiplication, for any
metric $h\in \cH(L)$, there is a unique function $f\in C^{\infty}_+(M)$ such
that $f\cdot h_0=h$. We claim that if $\Rtw_{(g,A,h_0)}>0$ for some triple
$(g,A,h_0)$, then $\Rtw_{(g,A,f\cdot h_0)}>0$ for any $f\in C^{\infty}_+(M)$.
To see this, observe that changing the metric does
not change the curvature $2$-form $\Omega$ or the Clifford algebra
bundle, though it does change the pointwise norm on spinors.  At the
same time, the {positivity conditions}
of Lemma \ref{lem:gpsccond} do not change
when the norm on $L$ is rescaled by the constant $f(x)$.
Thus we have proved the following:
\begin{proposition}
\label{prop:dontneedh}
Let $(M,L)$ be a spin$^c$ manifold with $\cR^{\tw}(M,L)\neq \emptyset$,
and let $h_0\in \cH(L)$ be a fixed Hermitian metric. Then the inclusion
$\cR^{\tw}(M,L)_{h_0}\subset \cR^{\tw}(M,L)$ is a homotopy equivalence.
\end{proposition}
\begin{remark}
From now on, we identify the spaces $\cR^{\tw}(M,L)$
and $\cR^{\gen}(M,L)$, and because of Proposition
\ref{prop:dontneedh} we will work with the space $\cR^{\tw}(M,L)$
assuming that the line bundle $L$ and the metric $h$ on $L$ are
fixed once and for all.  We will omit the metric $h$ from the
notations, i.e., work with pairs $(g,A)$ as elements of
$\cR^{\tw}(M,L)$ and $\cR^{\gen}(M,L)$.  Furthermore, we will see
soon that we really only need to keep track of the Riemannian metric
$g$ and not of the bundle connection $A$ since there is a canonical
choice of the connection which makes the curvature 2-form $\Omega$
harmonic with respect to the metric $g$.
\end{remark}
Next, we want to somewhat simplify the study of pairs $(g,A)\in \cR^{\tw}(M,L)$.
As a warm-up, let's consider just the usual psc
condition.  Let $M$ be a closed $n$-manifold admitting a psc metric,
$n\geq 3$.  Let $\cR(M)$ be the space of all Riemannian metrics on
$M$, and let $\cR^+(M)\subset \cR(M)$ be the subspace of psc metrics.
The space of (pointwise) conformal classes $\cC(M)$ is defined as the
orbit space of the left action
\begin{equation*}
  C^{\infty}_+(M)\times \cR(M) \to \cR(M), \ \ \ \ (f,g)\mapsto f \cdot g,
\end{equation*}
so that $\cC(M):=C^{\infty}_+(M)\backslash \cR(M)$ (here
$C^{\infty}_+(M)$ is the space of positive smooth functions). Denote
by $p\co \cR(M)\to \cC(M)$ the corresponding projection.
{We give $\cC(M)$ the quotient topology.  Note that $\cC(M)$
  is Hausdorff.  To see this, it suffices to show that if
  $g_n\in \cR(M)$, $f_n\in C^{\infty}_+(M)$, and $g_n\to g_0$,
  $f_ng_n\to g_1$, then $g_1=fg_0$ with $f_n\to f$ for some $f$ in
  $C^{\infty}_+(M)$. But this is clearly true pointwise, and since
  $\{g_n\}$ and $\{f_ng_n\}$ converge uniformly along with their derivatives,
  by definition of the topology on $\cR(M)$, it is also
  clear that one gets such convergence of $f_n$ to $f$.}   We say that
a conformal class $C\in \cC(M)$ is positive if it contains a
psc-metric, and let $\cC^+(M)\subset \cC(M)$ be the subspace of
positive conformal classes.  We note that the subspace
$\cR^{+,c}(M):=p^{-1}(\cC^+(M))$ of ``conformally positive Riemannian
metrics'' $g$ is strictly larger then the space $\cR^+(M)$.

According to Kazdan-Warner \cite[Theorem 3.2]{MR365409}, the
sign of the first eigenvalue $\mu_1(\sL_g)$ of the conformal Laplacian
\begin{equation}
\label{eq:confLap}
\sL_g=-\frac{4(n-1)}{n-2}\Delta_g + R_g
\end{equation}
is constant on conformal classes, and the space $\cR^{+,c}(M)$ can be
characterized as the set of metrics for which $\mu_1(\sL_g)$ is
strictly positive.  Indeed, if $\mu_1(\sL_g)>0$, then the
corresponding eigenfunction $\varphi$ can be chosen strictly positive
and \cite[Remark 2.12]{MR365409} the (pointwise) conformal metric
$g_1 = \varphi^{\frac{4}{n-2}}g$
has $R_{g_1}>0$ everywhere. Conversely, if $R_g>0$, then it is
clear that $\sL_g$ is bounded away from zero by a positive constant
(the minimum value of $R_g$), and thus $\mu_1(\sL_g)>0$.
So $\cR^+(M)\subseteq \cR^{+,c}(M)$
and $\cR^{+,c}(M)$ is the smallest family of Riemannian metrics
containing $\cR^+(M)$ and closed under (pointwise) conformal changes of the
metric.  On spin manifolds, invertibility of the Dirac operator
extends from $\cR^+(M)$ to $\cR^{+,c}(M)$ via \cite{MR834486},
which says that $\lambda_1^2(D_g)\ge\frac{n}{4(n-1)}\mu_1(\sL_g)>0$, where
$\lambda_1(D_g)$ is the first eigenvalue of the Dirac operator $D_g$.
\begin{theorem}
\label{thm:confclass}
Let $M$ be a closed manifold, $\dim M = n\geq 3$, admitting a metric
of positive scalar curvature.  The inclusion $\cR^+(M)\hookrightarrow
\cR^{+,c}(M)$ is a homotopy equivalence.
\end{theorem}  
\begin{proof}
Note that $\cR^+(M)$ and $\cR^{+,c}(M)$ are both open subsets of {the
  separable Fr\'echet space of symmetric $2$-tensors on the tangent
  bundle}, so they {each} have the homotopy type of a countable CW
complex \cite[Theorem 1]{MR100267}. We will show that the projection
maps $\cR^+(M)\to {\cC^+(M)}$ and $p\co\cR^{+,c}(M)\to {\cC^+(M)}$ are
both fibrations with contractible fibers, and hence are homotopy
equivalences. By the definitions, both of these maps are surjective,
and $\cC^+(M)$ has the quotient topology.  Since $C^\infty_+(M)$ is a
contractible Fr\'echet Lie group (under multiplication)
and it acts smoothly and freely on
$\cR^{+,c}(M)$, with {$\cC^+(M)$} the (paracompact) quotient space,
the projection map $p\co\cR^{+,c}(M)\to \cC^+(M)$ is a principal
bundle for a contractible Fr\'echet Lie group, and thus is a homotopy
equivalence. In particular, it has a section. Choose such a
section $s_1\co \cC^+(M)\to \cR^{+,c}(M)$; it's not canonical, but
that won't matter for us.  So for each conformal class $C\in
\cC^+(M)$, $s_1(C)$ is a metric in the class $C$ for which
$\mu_1(\sL_{s_1(C)})>0$. There is a positive eigenfunction $\varphi_C$
for the eigenvalue $\mu_1(\sL_{s_1(C)})$, and it's unique up to
multiplication by positive scalars.  We can make it unique (on the
nose!) by normalizing it to have maximum value $1$ on $M$.  Then the
map $C\mapsto \varphi_C$ is continuous in the appropriate topologies,
because of elliptic regularity.  Now we get a section $s\co
\cC^+(M)\to \cR^+(M)$ by taking $s(C) = \varphi_C^{\frac{4}{n-2}}
s_1(C)$, using \cite[Remark 2.12]{MR365409} to confirm that the scalar
curvature of this metric is indeed strictly positive.  We need to show
that $\cR^+(M)\to \cC^+(M)$ is also a fibration with contractible
fibers. We will show not only that the fibers are contractible
open subsets of the fibers of $p$, but also that there is a
deformation retraction from $\cR^+(M)$ down to $s(\cC^+(M))$ (a
homeomorphic copy of $\cC^+(M)$), and thus that $\cR^+(M)$ and
$\cC^+(M)$ are homotopy equivalent.

Fix $C\in \cC^+(M)$
and some metric $g_0\in p^{-1}(C)\cap \cR^+(M)$.  We will show that
$p^{-1}(C)\cap \cR^+(M)$ is convex in a suitable sense, hence
also contractible. 

Let $g_1\in p^{-1}(C)\cap \cR^+(M)$.  Then there is a homotopy
$g_t,\,0\le t\le 1$ from $g_0$ to $g_1$ within $p^{-1}(C)$, given by
$g_t = u_t^{\frac{4}{n-2}}g_0$, where $u_t = (1-t)+tu$, for the unique
$u\in C^\infty(M)$ with $u>0$ everywhere and $u^{\frac{4}{n-2}}g_0=g_1$.
This $u$ depends continuously on $g_1$. We just need to observe that each
$g_t$ has psc.  But by the formula for scalar curvature of a conformal
deformation of a metric, $R_{g_t}=u_t^{-\alpha}\sL_{g_0}(u_t)$, where
$\alpha = 1+\frac{4}{n-2}$.  But by linearity of the operator
$\sL_{g_0}$, we have $\sL_{g_0}(u_t)=(1-t)\sL_{g_0}(1)+t\sL_{g_0}(u)$.
Now $\sL_{g_0}(1)=R_{{g_0}}>0$ and $\sL_{g_0}(u)=u^{\alpha}R_{g_1}$ is also
positive, since we assumed $g_1\in \cR^+(M)$.  Thus $R_{g_t}>0$ for
all $0\le t\le 1$, as we needed to show.

Now we get a deformation retraction of $\cR^+(M)$
down to $s(\cC^+(M))$ by using the above homotopy from a general
metric $g_1\in \cR^+(M)$ to $g_0 = s(p(g_1))$, and thus $p$ restricted
to $\cR^+(M)$ is a homotopy equivalence.
\end{proof}
The advantage of Theorem \ref{thm:confclass} is that it makes it
possible to study the topology of $\cR^+(M)$ via tools involving
conformal geometry.

Now we would like to do something similar for the curvature condition
$\Rgen_{(g,A)}>0$ on spin$^c$ manifolds (when $n\ge3$).
For this the paper
\cite{MR834486}, giving estimates on eigenvalues of the Dirac operator
using the conformal Laplacian, should be replaced by
\cite{MR1162671,MR1705917,MR2661160}.  
What B\"ar calls the curvature endomorphism $\cK$ is
$\frac{1}{4}R_g+\frac12 i c(\Omega) = \frac14 \Rtw_{(g,A)}$, which is
a self-adjoint endomorphism at each point of $M$, and he takes
$\kappa$ \cite[bottom of p.\ 42]{MR1162671} to be a scalar lower bound
for $\cK$, which since the eigenvalues of $ic(\Omega)$ come in $\pm$
pairs can be taken to be $\frac14 \Rgen_{(g,A)}$, so the estimate from
\cite[Theorem 3]{MR1162671} turns out to be exactly the inequality
\begin{equation*}
\lambda_1^2(D_L)\ge \frac{n}{4(n-1)}\mu_1(\sL_{(g,A)}),
\end{equation*}
with $\sL_{(g,A)}$ defined by replacing $R_g$ by $\Rgen_{(g,A)}$ in
the formula for $\sL_g$, and this 
is nothing but the spin$^c$ analogue of Hijazi's inequality.

In some cases, one can basically get rid of the connection $A$ on
the line bundle $L$ for the condition $\Rgen_{(g,A)}>0$.
Here is a case where this is possible.
\begin{lemma}
\label{lem:2normcurvature}
Let $L$ be a line bundle on a closed Riemannian manifold $(M,g)$.
Then the $2$-norm\footnote{This is defined by $\Vert\Omega\Vert_2^2 =
\int_M \Omega\wedge *\Omega$.}  of the curvature $2$-form $\Omega$ of
a connection on $L$ is minimized when $\Omega$ is the unique harmonic
$2$-form in its de Rham class.
\end{lemma}
\begin{proof}
  This fact is well-known but we recall the proof.  {Recall the
    de Rham class of the form $\Omega$ (representing  the first Chern
    class $c_1(L)=[\frac{1}{2\pi}\Omega]$ in $H^2(M,\bR)$)}
    contains a unique harmonic $2$-form $\Omega_0$, which (up to a
    factor of $2\pi$) is the curvature of a connection $A_0$ on $L$.
    Any other connection on $L$ is of the form $A_0+\alpha$, where
    $\alpha$ is an ordinary $1$-form, and has curvature $2$-form
    $\Omega=\Omega_0+d\alpha$.  We have for any $1$-form $\alpha$
\begin{equation*}
\langle \Omega_0,d\alpha\rangle = \langle d^*\Omega_0,\alpha\rangle
=0
\end{equation*}
since $d^*\Omega_0=0$ because $\Omega_0$ is harmonic.  Thus, by the geometry of inner
product spaces,
\begin{equation*}
\Vert \Omega\Vert_2^2=\Vert \Omega_0\Vert_2^2 +\Vert d\alpha\Vert_2^2
\end{equation*}
is minimized when $d\alpha=0$ and $\Omega = \Omega_0$ is harmonic. {(See
also  \cite[Lemma D.13]{lawson89:_spin}}.)
\end{proof}
\begin{remark}
  Clearly the condition $\Rgen_{(g,A)} = R_g - 2|\Omega|_{op}>0$ is not
directly a condition on the $2$-norm of $\Omega$.  However,
by \cite[Lemma 3.3]{MR1705917} and the fact that Clifford
multiplication by $\Omega$ is skew-symmetric,
one has the pointwise estimate
$|\Omega|_{op}\le \sqrt{\frac{n}{2}}|\Omega|_2$.  In the other direction,
one has a similar estimate
\begin{equation*}
|\Omega|_2^2 = \frac12 \Tr c(\Omega)^* c(\Omega)
\le \frac{n}{2}|\Omega|_{op}^2
\end{equation*}
and thus $|\Omega|_2\le \sqrt{\frac{n}{2}}|\Omega|_{op}$.
If $\Omega$ is harmonic, then $\Omega$ {has constant coefficients}
in {a local harmonic coordinate system}, and
thus $|\Omega|_2$ and $|\Omega|_{op}$ are constant
functions.  Furthermore, one always has
\begin{equation*}
\Vert \Omega\Vert^2_2 = \int_M |\Omega|_2^2\,d\vol,
\end{equation*}
which when $\Omega$ is harmonic becomes $\vol(M)\,|\Omega|_2^2$.
Thus in this case $\Vert\Omega\Vert_2=\vol(M)^{\frac12}|\Omega|_2$,
and if $\Vert\Omega\Vert_2$ is small enough, one can
use Lemma \ref{lem:2normcurvature} to replace the original
generalized psc connection by one with harmonic curvature. But the
estimates comparing $|\Omega|_{op}$ and $|\Omega|_2$ get worse and
worse as $n\to \infty$.
\end{remark}
Now we proceed to try to find an analogue of Theorem
\ref{thm:confclass}.  Let $(M,L)$ be a closed spin$^c$ manifold of
dimension $n\geq 3$ and $\cR^{\gen}(M,L)\neq \emptyset$. Let $\Omega$
be the curvature $2$-form on $A$. Define the modified conformal
Laplacian of $(M,L)$ equipped with a Riemannian metric $g$ and a
connection $A$ to be
\begin{equation}
\label{eq:twistedconfLap}
\sL_{(g,A)}:=-\frac{4(n-1)}{n-2}\Delta_g + \Rgen_{(g,A)} = \sL_g
- 2|\Omega|_{op}.
\end{equation}
Note of course that if $M$ is spin, the line bundle $L$ is trivial and
we can take $\Omega=0$, so then $\Rgen_{(g,A)} = R_g$ and
$\sL_{(g,A)}$ is just the usual conformal Laplacian $\sL_g$.  By
\cite[Theorem 3]{MR1162671} (even in the general spin$^c$ case), the
lowest eigenvalue $\mu_1(\sL_{(g,A)})$ of $\sL_{(g,A)}$ gives a lower
bound for the smallest eigenvalue of the square of the spin$^c$ Dirac
operator $D_L$:
\begin{equation}\label{eq:bar-estimate}
\lambda_1^2(D_L)\ge \frac{n}{4(n-1)}\mu_1(\sL_{(g,A)}).
\end{equation}
Let's show that $\Rgen_{(g,A)}\ge 0$ with $\Rgen_{(g,A)}$ positive somewhere
on $M$ implies that $\mu_1(\sL_{(g,A)})>0$.  The argument is the same
as for $\sL_g$.  Suppose $\Rgen_{(g,A)}\ge 0$ and $\Rgen_{(g,A)}$ is
positive somewhere.  Clearly $\sL_{(g,A)}\ge 0$, so that
$\mu_1(\sL_{(g,A)})\ge 0$. But suppose $\mu_1(\sL_{(g,A)})= 0$.
There is always positive eigenfunction for $\mu_1(\sL_{(g,A)})$,
so we have a strictly positive function $\phi$ in the kernel of
$\sL_{(g,A)}$.  We have
\[
0=\langle \sL_{(g,A)}\phi,\phi\rangle =
\frac{4(n-1)}{n-2}\int_M \langle \nabla\phi, \nabla \phi\rangle\,d\vol
+ \int_M \Rgen_{(g,A)}\phi^2\,d\vol,
\]
and both terms on the right are $\ge 0$, so they both vanish.
Vanishing of the first term shows that $\phi$ is a positive constant,
and then vanishing of the second term shows that $\Rgen_{(g,A)}\equiv 0$,
a contradiction.
{In particular, the inequality (\ref{eq:bar-estimate})
  impl{ies} that the Dirac operator $D_L$ is invertible.}

We denote by $\cR^{\gen, c}(M,L)$ the space of
pairs $(g,A)\in \cR(M)\times \cA(M,L)$ such that
$\mu_1(\sL_{(g,A)})>0$.  Now we consider what happens under pointwise
conformal deformation of the metric $g$. Note
that we don't change $h$ or $A$, so $\Omega$ is fixed, though the norm
$|\Omega|_{op}$ of $\Omega$ varies with metric, since it depends on
the Clifford algebra.  
\begin{lemma}
\label{lem:confinv2}
Let $(M,L)$ be a closed spin$^c$ manifold, $\dim M = n\geq 3$, and let
$(g,A)$ be a pair such that $\mu_1(\sL_{(g,A)})>0$. Then there exists a
conformal metric $g_1=u^{\frac{4}{n-2}}g$ such that $\Rgen_{(g_1,A)}>0$. Moreover,
the space $\cR^{\gen, c}(M,L)$ is invariant under pointwise conformal
changes of the metric, and for $(g,A)\in \cR^{\gen, c}(M,L)$, there is
a pointwise conformal metric $g_1$ with $\Rgen_{(g_1,A)}>0$
\end{lemma}
\begin{proof}
  Let $u\in C^\infty(M)$, $u>0$, and let $g_1 = u^{\frac{4}{n-2}}g$.  We need
  to show that $\Rgen_{(g_1,A)}>0$. We have
  $R_{g_1}=u^{-\alpha}\sL_{g}(u)$, where $\alpha = 1+\frac{4}{n-2}$.
  If we rescale the metric by a constant $c>0$, the curvature is
  multiplied by $c^{-1}$, and $|\Omega|_{op}$ scales the same
  way as the curvature, so that computed with respect to the new metric
  $g_1$, $|\Omega|_{op}$ is multiplied by $u^{-\frac{4}{n-2}}$.  {Then we use
\eqref{eq:twistedconfLap} to compute $\Rgen_{(g_1,A)}$:}    
\begin{equation*}
  \begin{aligned}
    \Rgen_{(g_1,A)}=R_{g_1}-2|\Omega|_{op,g_1}
    &= u^{-\alpha}\sL_{g}(u)-2u^{-\frac{4}{n-2}}|\Omega|_{op,g}\\
    &=-\frac{4(n-1)}{n-2}u^{-\alpha}\Delta_g(u) +
    u^{-\alpha}R_gu - 2uu^{-\alpha}|\Omega|_{op,g}\\
    &=u^{-\alpha}{\sL_{(g,A)}}(u) . 
  \end{aligned}
\end{equation*}
Finally, suppose $(g,A)\in \cR^{\gen, c}(M,L)$, and let $\varphi$ be a strictly
  positive eigenfunction for the first eigenvalue $\mu_1(\sL_{(g,A)})>0$.
  Replace $g$ by the pointwise conformal metric $g_1=\varphi^{\frac{4}{n-2}}g$.
  Then
\begin{equation*}
  \begin{aligned}
    \Rgen_{(g_1,A)}=R_{g_1}-2|\Omega|_{op,g_1}
    &= \varphi^{-\alpha}\sL_{g}(\varphi)-2\varphi^{-\frac{4}{n-2}}|\Omega|_{op,g}\\
    &=-\frac{4(n-1)}{n-2}\varphi^{-\alpha}\Delta_g(\varphi) +
    \varphi^{-\alpha}R_g\varphi - 2\varphi\varphi^{-\alpha}|\Omega|_{op,g}\\
    &=\varphi^{-\alpha}\sL_{(g,A)}(\varphi)=\mu_1(\sL_{(g,A)})\varphi\varphi^{-\alpha}>0.
  \end{aligned}
\end{equation*}
  This concludes the proof.
\end{proof}

Now we have an analogue of Theorem \ref{thm:confclass}.
\begin{theorem}
\label{thm:gpscconfclass}
Let $(M,L)$ be a closed spin$^c$ manifold, $n\geq 3$, such that
$\cR^{\gen, c}(M,L)\neq \emptyset$. Then the inclusion
$\cR^{\gen}(M,L)\subset\cR^{\gen, c}(M,L)$ is homotopy equivalence.
\end{theorem}  
\begin{proof}
  We proceed as in the proof of Theorem \ref{thm:confclass}.  Let
  $(g,A)\in \cR^{\gen}(M,L)$, i.e., $\Rgen_{(g,A)}=R _g - 2|\Omega|_{op}>0$,
  and let $g_1=u^{\frac{4}{n-2}}g_0$ be pointwise conformal to $g$,
  for some $u\in C^\infty_+(M)$. We have a homotopy $u_t = (1-t)+tu$,
  $g_t=u_t^{\frac{4}{n-2}}g_0$, $0\le t\le 1$, where $u_t\in
  C^\infty_+(M)$. By the formula for scalar curvature of a conformal
  deformation of a metric, $R_{g_t}=u_t^{-\alpha}\sL_{g_0}(u_t)$,
  where $\alpha = \frac{n+2}{n-2}$.  But by linearity of $\sL_{g_0}$,
  $\sL_{g_0}(u_t)=(1-t)\sL_{g_0}(1)+t\sL_{g_0}(u) =
  (1-t)R_g+t\sL_{g_0}(u)$.  So
\begin{equation*}
  \begin{aligned}
  \Rgen_{(g_t,A)}=R_{g_t} - 2|\Omega|_{op,g_t} &=
  u_t^{-\alpha}(1-t)R_g+u_t^{-\alpha}t\sL_{g_0}(u)-2|\Omega|_{op,g}u_tu_t^{-\alpha}\\
  &=u_t^{-\alpha}\bigl((1-t)R_g + t\sL_{g_0}(u) -2|\Omega|_{op,g}(1-t)
  -2|\Omega|_{op,g}tu\bigr)\\
  &= u_t^{-\alpha}\bigl((1-t)\Rgen_{(g,A)} + t\sL_{(g_0,A)}(u)\bigr)>0,
  \end{aligned}
\end{equation*}
  where $\kappa=\Rgen_{(g_0,A)}$.
  This shows that the fibers of the projection to pointwise conformal
  classes are convex, and so the projection map is a homotopy
  equivalence{, by exactly the same argument as in the proof of
    Theorem \ref{thm:confclass}}.
\end{proof}

Now we can effectively get rid of having to worry about the connection
on the line bundle by the following trick.  Start with a pair $(g,A)$
with $\Rgen_{(g,A)}>0$.  By Theorem \ref{thm:gpscconfclass} and the
solution of the Yamabe problem, we can make a pointwise conformal
change of the metric $g$ to a Riemannian metric $\check{g}$ with
constant positive scalar curvature, and still stay in the space
$\cR^{\gen, c}(M,L)$.  Now replace $g$ by $\check{g}$.  The curvature
$2$-form $\Omega$ must have the form $\Omega_0 + d\alpha$, where
$\Omega_0$ is the unique \emph{harmonic} $2$-form (with respect to the
metric $\check{g}$) in the de Rham class of $\Omega$.  In harmonic
coordinates (with respect to the metric $\check{g}$), $\Omega_0$ is a
$2$-form with constant coefficients, so its pointwise norm
$|\Omega_0|_{op,\check{g}}$ is a constant function.  Thus if
$|\Omega|_{op,\check{g}}$ is \emph{not} constant, replacing $\Omega$
by $\Omega_0$ (which corresponds to replacing the initial connection
$A$ by a connection $A_0$) lowers the maximum value of the pointwise
operator norm, and replaces $\Rgen_{(\check{g},A)}$ by a constant
$\Rgen_{(\check{g},A_0)}= R_{\check{g}} - 2|\Omega_0|_{op,\check{g}}$,
thereby increasing its minimum value.  Since $\sL_{(\check{g},A_0)}$ is
positive, this constant must be positive.

Thus it is no loss
of generality to assume that the curvature $2$-form is harmonic, and thus
determined by its Chern class and the Riemannian metric.
{%
\begin{remark}
We note that the aforementioned argument is not applicable to an
arbitrary family of initial metrics, $g_{\upsilon}$, $\upsilon\in\Upsilon$,
because the Yamabe problem generally does not admit a unique solution.
\end{remark}  
}
\section{A Trichotomy Theorem}
\label{sec:trichotomy}
In the last section, we saw that there is a close analogy between
generalized scalar curvature on a spin$^c$ manifold and scalar curvature
in the spin case.
We state again a trichotomy theorem for spin$^c$ manifolds.
Our corresponding theorem is the following:
\begin{theorem}
\label{thm:trichotomy}
  Every closed connected spin$^c$ manifold
  $(M,L)$ of dimension at least three falls into exactly one of the
  following three classes:
  \begin{enumerate}
  \item those admitting a pair $(g,A)$ with {$\Rgen_{(g,A)}>0$}, in which
    case every smooth function on $M$ is the generalized scalar
    curvature of some pair $(g',A')$, $g'$ a Riemannian metric on $M$ and
    $A'$ a connection on $L$;
  \item those admitting a pair $(g,A)$ with {$\Rgen_{(g,A)}\ge 0$} but not
    a pair $(g,A)$ with {$\Rgen_{(g,A)}$} positive everywhere --- in this
    case a smooth function on $M$ is
    the generalized scalar curvature of some pair $(g',A')$ if and
    only if it's either identically $0$ or else negative somewhere;
  \item all other manifolds, those not admitting any $(g,A)$ with
    {$\Rgen_{(g,A)}\ge 0$} --- in this case, a smooth function on $M$ 
    is the generalized scalar curvature of some pair $(g',A')$ if and only if
    it is negative somewhere on $M$.  
  \end{enumerate}
\end{theorem}
Note that in all three cases, any smooth function which is negative somewhere
can be { a {generalized scalar} curvature $\Rgen_{(g,A)}$}
for some $(g,A)$.  This is 
Theorem \ref{ref:neggsc} below.  Rigidity aspects of case (2) of
Theorem \ref{thm:trichotomy} will be discussed further in section
\ref{sec:rigidity}.  The proof of Theorem \ref{thm:trichotomy} will be
broken up in several steps.  The first step, helping to distinguish
between cases (1) and (2), was already mentioned in
section \ref{sec:genpsc}: if $\Rgen_{g,A}\ge 0$ and is not identically
$0$, then $\mu_1(\sL_{(g,A)})>0$, and so by pointwise conformal
deformation of the metric from $g$ to $g'$, one can arrange to make
$\Rgen_{(g',A)}>0$ everywhere.  The next step is the following,
which clearly takes care of statements \textup{(2)} and \textup{(3)} in
Theorem \textup{\ref{thm:trichotomy}}.  The argument here is
closely modeled on \cite{MR365409}.
\begin{theorem}
\label{ref:neggsc}
Let $(M,L)$ be a connected {closed} spin$^c$ manifold
of dimension $n\ge 3$ and let $\kappa$ be a smooth function on $M$
taking a negative value somewhere on $M$.  Then $\kappa$ is the
generalized scalar curvature of some pair $(g,A)$, $g$ a
Riemannian metric on $M$ and $A$ a connection on $L$.
\end{theorem}
\begin{proof}
  By \cite[Lemma 3.6 and Theorem 4.1]{MR365409} or
  \cite[Lemma 3(a)]{MR0375154}, $M$ always has a metric $g_0$ with constant
  negative scalar curvature.  Choose any connection $A$ on $L$.
  Then $\Rgen_{(g_0,A)}$ is strictly negative everywhere.  Consider the
  twisted conformal Laplacian \eqref{eq:twistedconfLap}
  $\sL_{(g_0,A)}$.  The lowest eigenvalue $\mu_1(\sL_{(g_0,A)})$ of
  this operator must be negative, since we have
\begin{equation*}
  \langle \sL_{(g_0,A)}(1), 1\rangle = \int_M\Rgen_{(g_0,A)}\,d\vol_{g_0}
  < 0.
\end{equation*}
  (In fact the same calculation shows that
\begin{equation*}
  \int_M\Rgen_{(g,A)}\,d\vol_g < 0
\end{equation*}
  for any $(g,A)$ will imply that $\mu_1(\sL_{(g_0,A)})<0$.)

{In order to make the arguments below transparent, let us briefly
  recall a relevant result by Kazdan-Warner from \cite[Section 2]{MR365409}. 
  They consider a general equation}
  \begin{equation}\label{eq:KW2.2}
{    -\mathbf{L}u = \Delta u - h u = - Hu^{\alpha} \ ,}
  \end{equation}
  {where $h$ and $H$ are prescribed functions and $\alpha>1$ is a
    constant. In our case we assume $h$ and $H$ are both smooth. We
    denote by $\mu_1(\mathbf{L})$ the first eigenvalue of
    $\mathbf{L}$, and by $u_1>0$ the corresponding eigenfunction. One
    can observe that if there exists a solution $u>0$ of
    (\ref{eq:KW2.2}), then the signs of $\mu_1(\mathbf{L})$ and $H$
    are the same, provided $H$ is {never} zero; see \cite[Lemma
      2.5]{MR365409}. Then the following result holds:
  \begin{proposition}\label{propK-W2.11}{\rm \cite[Theorem 2.11(b)]{MR365409}}
    Assume {that} $H<0${; then a}
    positive solution of \eqref{eq:KW2.2} exists if and
only if $\mu_1(\mathbf{L})<0$. 
  \end{proposition}  
The idea to construct the solution $u$ is to use \emph{upper} and
\emph{lower} solutions $u_+$ and $u_-$. In the case {at hand}, $u_+=$
large constant, and $u_-=\epsilon u_1$, where $u_1$ is the first
eigenfunction: $\mathbf{L}u_1=\mu_1(\mathbf{L})u_1$.}
  \vspace{3mm}
  
Now we continue with the proof of Theorem \ref{ref:neggsc}.  First
suppose that {the function} $\kappa$ is everywhere
{negative}. By \cite[Theorem 2.11(b)]{MR365409} {(or by
  Proposition \ref{propK-W2.11})}, there is a {smooth function}
$u>0$ such that with $g=u^{\frac{4}{n-2}}g_0$, $(g, A)$ has
$\Rgen_{(g,A)}=\kappa$, since by the calculation in the proof of Lemma
\ref{lem:confinv2}, if we let $\alpha=\frac{n+2}{n-2}$, then we need
to solve {the equation $L_{(g_0,A)} (u) = \kappa u^{\alpha}$,
  which is the same as}
\begin{equation*}
  \Rgen_{(g,A)}=u^{-\alpha}L_{(g_0,A)} (u) = \kappa.
\end{equation*}
Then the existence of the solution $u$ follows from Proposition
\ref{propK-W2.11}.

  The case where $\kappa$ is allowed to be {non-negative} on some
  subset of $M$ is more delicate, since a pointwise conformal
  deformation of $g_0$ may not suffice. {In} this case we show
  {first} that after moving $\kappa$ around by a diffeomorphism{,
    we can make it} very negative on a ``large'' portion of $M$.

{Let $g_0$ be the same metric on $M$ with constant negative scalar
curvature. We choose a diffeomorphism $\phi\co M\to M$ such that 
\begin{equation}\label{eq:int-negative}
\int_M\phi^*(\kappa)\,d\sigma_{g_0} < 0.  
\end{equation}
This will suffice since if
  $\Rgen_{(g,A)}=\phi^*(\kappa)$ for some diffeomorphism $\phi$ of
  $M$, then
\begin{equation*}
  \Rgen_{\bigl((\phi^{-1})^*g,(\phi^{-1})^*(A)\bigr)}=\kappa.  
\end{equation*}
It is not so difficult to arrange a diffeomorphism $\phi$ such that
the inequality \eqref{eq:int-negative} holds. Choose a constant $C_->0$
such that $\kappa(x)<-C_-$ for some $x\in M$, and find a diffeomorphism
$\phi$ such that the image $\phi(\{\kappa < -C_-\})$ has almost the same
$g_0$-volume as the $g_0$-volume of $M$. If, say,
\begin{equation*}
\Vol_{g_0}(M)-\Vol_{g_0}(\phi(\{\kappa < -C_-\}))=\epsilon
\end{equation*}
and $\kappa\leq C_+$, then 
\begin{equation}\label{eq:int-negative1}
\int_M\phi^*(\kappa)\,d\sigma_{g_0} < (-C_-+\epsilon \cdot C_+)\Vol_{g_0}(M) <0  
\end{equation}
for small enough $\epsilon>0$.  We note that the condition
(\ref{eq:int-negative}) means that the function $\phi^*(\kappa)$ is
``not too positive too often'', see \cite[p.\ 124 and Lemma
  2.15]{MR365409}. This fact, combined with the condition that
$\mu_1(\sL_{(g_0,A)}) < 0$, guarantees the existence of a smooth,
positive function $u$ satisfying the equation $\sL_{(g_0,A)}u =
\phi^*(\kappa)u^{\alpha}$. Thus we obtain
$\Rgen_{(g,A)}=\phi^*(\kappa)$ for suitable $\kappa$ and $u>0$, with
$g=u^{\frac{4}{n-2}}g_0$, follows by the same continuity argument as
in \cite[\S4]{MR365409}.}
\end{proof}
%
\begin{proof}[End of the proof of {Theorem \ref{thm:trichotomy}}]
To finish the proof of Theorem \ref{thm:trichotomy}, we need to show
that if $\kappa\ge 0$ and $(M,L)$ admits a pair $(g,A)$ with
$\Rgen_{(g,A)}>0$, then $\kappa$ can be realized as the generalized
scalar curvature of some pair $(g',A')$ on $(M,L)$.  (The case where
$\kappa$ is negative somewhere is already covered by Theorem
\ref{ref:neggsc}.)

By the argument at the end of section \ref{sec:genpsc}, the assumption
implies that $(M,L)$ admits a pair $(\check{g},A)$ with a constant
$R_{\check{g}}$, and there is a connection $A$ on $M$ with curvature
form $\Omega$ with $\vert\Omega\vert_{op,\check{g}}$ constant so that
$\Rgen_{(\check{g},A)}$ is a positive constant. By rescaling, we can
obviously arrange for any positive constant to be realized as the function
$\Rgen_{(g',A')}$ for some pair $(g',A')$.  {To show that one can
  achieve $\Rgen_{(g',A')}\equiv 0$ for some pair $(g',A')$, just
  argue as in the proofs of \cite[Theorems 3.9 and 1.3]{MR0375153}.}
So we just need to deal
with the case where $\kappa\ge0$ is not constant.  Choose $(g',A')$ as
just explained so that $\Rgen_{(g',A')}$ is a positive constant. Since
$\kappa$ is assumed $\ge 0$ and nonconstant, we can choose a constant
$c>0$ so that $0\le \min \kappa <c\Rgen_{(g',A')} < \max \kappa$.  Then we
argue as in \cite[Lemma 6.1]{MR0375153}.
\end{proof}

\section{Rigidity Theorems}
\label{sec:rigidity}
In this section, we go in more detail into the
case (2) of Theorem \ref{thm:trichotomy}.  Recall that
case (2) of Theorem \ref{thm:KW-int}
involves a rather strong rigidity statement, as it implies that a
metric on $M$ in class (2) with $R_g\ge 0$ {is necessarily Ricci-flat}.
We will see that something similar happens in our situation as well.

The analogue of the Kazdan-Warner case (2) in our situation is the case
of a simply connected non-spin spin$^c$ manifold $(M,L)$ with
$\alpha^c(M,L)\ne 0$ (this of course means the Dirac operator $D_L$
has to have a nontrivial kernel) but with a pair $(g,A)$ for which
$\mu_1(L_{(g,A)})=0$.

The simplest example of this situation is $M=\bC\bP^2$ with the
Fubini-Study metric {$g_{FS}$}, $L=\cO(3)$ (this is the smallest
ample line bundle with odd Chern class and nonvanishing
$\alpha^c(M,L)$, and is also the anti-canonical bundle), and the
standard connection {$A$} on $L$ with curvature $3\omega$, where
$\omega$ is the K\"ahler form.  {Since the scalar curvature
  $R_{g_{FS}}=6$, then $\Rgen_{(g_{FS},A)}=R_{g_{FS}}-2\cdot
  3|\omega|_{op}\equiv 0$  (see also the calculation in the proof of
  \cite[Corollary 5.2]{BJ}).}
The next theorem shows that in
some sense, this example is typical for this unusual situation.
\begin{theorem}
\label{thm:specialhol}
Let $(M,L)$ be a simply connected spin$^c$ manifold which is non-spin,
and assume that $\alpha^c(M,L)\ne 0$.  Suppose there is a pair $(g,A)$
{\lp}a Riemannian metric $g$ on $M$ and a connection $A$ on $L${\rp}
for which $\mu_1(L_{(g,A)})=0$.  Then $(M,g)$ is pointwise conformal to 
another metric $g'$ for which $\Rgen_{(g',A)}$ vanishes identically.
Moreover, then $(M,L)$ with metric $g'$ admits a parallel spinor, and
$(M,g')$ is a Riemannian product of a K\"ahler manifold and a simply
connected spin manifold with a parallel spinor.  In particular, if $M$
does not split as a product, then $M$ is conformally K\"ahler, and $L$
is either the canonical or the anti-canonical line bundle on $M$.
\end{theorem}
\begin{proof}
  We apply the argument in the proof of Lemma \ref{lem:confinv2}.
  Let $\varphi$ be a strictly
  positive {eigenfunction} for the first eigenvalue $\mu_1(L_{(g,A)})=0$.
  Replace $g$ by the pointwise conformal metric $g'=\varphi^{\frac{4}{n-2}}g$.
  The calculation in the proof of Lemma \ref{lem:confinv2} gives that
  $\Rgen_{(g',A)}\equiv 0$, and the twisted conformal Laplacian for
  this new metric is just a multiple of the usual Laplacian on functions,
  so its lowest eigenvalue is of course still $0$ (with eigenfunction $1$).
  Replacing $g$ by $g'$, we now have
  $D_L^2=\nabla^*\nabla + \Rtw_{(g',A)}$, where $\Rtw_{(g',A)}\ge 0$
  everywhere, with equality on the $0$-eigenspace of $c(\Omega)$.
  Suppose $\psi$ is a harmonic spinor.  (These exist since 
  $\alpha^c(M,L)\ne 0$.) We have
\begin{equation*}
  0=\langle D_L\psi,D_L\psi\rangle = \Vert \nabla\psi\Vert^2 + \frac14
  \langle \Rtw_{g_1,A}\psi,\psi\rangle,
\end{equation*}
  and since both terms on the right are nonnegative, they both
  have to vanish.  Thus any
  harmonic spinor is necessarily parallel. The last statement follows from
  \cite[Theorem 1.1]{MR1463835}.
\end{proof}
\begin{corollary}
  Suppose $(M,L)$ is a simply connected spin$^c$ manifold which is not
  spin.  Assume that $M$ does not split as a product, does not admit
  a K\"ahler metric, and satisfies $\alpha^c(M,L)\ne 0$.  Then for
  any choice of $(g,A)$ on $(M,L)$, the lowest eigenvalue $\mu_1$ of
  the twisted conformal Laplacian must be negative.
\end{corollary}
\begin{proof}
  If $\mu_1(L_{(g,A)})$ were positive, then by Lemma
  \ref{lem:confinv2}, we could change the metric so that $D_L^2$ is
  strictly positive, which of course would imply that
  $\alpha^c(M,L)=0$, a contradiction.  And if $\mu_1(L_{(g,A)})$ were
  $0$, Theorem \ref{thm:specialhol} would apply, which again would
  give a contradiction.  Thus $\mu_1(L_{(g,A)})<0$.
\end{proof}
\begin{example}
  For example, suppose $M=\bC\bP^{2n}\#\bC\bP^{2n}$.  This manifold does
  not split as a product and does not even admit an almost complex
  structure \cite{MR3936016}, so certainly it is not K\"ahler for any metric.
  The spin$^c$ line bundles $L$ on $M$ are parameterized by pairs
  $(k,\ell)$ of odd integers (representing the two components of
  $c_1(L)$), and
\begin{equation*}
  \alpha^c(M,L)=\alpha^c(\bC\bP^{2n},\cO(k))+\alpha^c(\bC\bP^{2n},\cO(\ell))
\end{equation*}
  can only vanish if $|k|$ and $|\ell|$ are both less than $2n+1$
  \cite{MR508087}. So if one or the other of $k$ and $\ell$ is
  at least $2n+1$ in absolute value, $\mu_1$ must be negative for any
  choice of $(g,A)$ on $(M,L)$.
\end{example}
\begin{example}
  \label{example:Enr}
  Let $M$ be an Enriques surface, that is a complex surface
  double-covered by a K3 surface.  Since a K3 surface has
  $\widehat A$-genus $2$, $M$ has $\widehat A$-genus $1$,
  which contradicts Rokhlin's theorem, and
  thus $M$ cannot be spin.  However, it is spin$^c$, with a canonical
  spin$^c$ line bundle $L$ for which $c_1(L)$ is the unique non-zero
  class in $H^2(M, \bZ)_{\text{tors}}\cong \bZ/2$.  It is also
  well-known (by Yau's theorem) that $M$ admits a Ricci-flat
  K\"ahler-Einstein metric $g$.  Since $L$ admits a flat connection
  $A$ (since $L$ comes from the fundamental group of $M$), we have
  $\Rgen_{(g,A)} \equiv 0$, and $M$ belongs to class (2) in Theorem
  \ref{thm:trichotomy}.  Of course, this $M$ is not simply connected.
\end{example}
We also observe that there is a partial converse to Theorem
\ref{thm:specialhol} that applies to K\"ahler-Einstein Fano varieties.
Once again, this generalizes our previous example of
$(\bC\bP^2,\cO(3))$.
A very similar formulation appears in {\cite{MR1876863}}.
\begin{theorem}{\rm ({cf.\ }Goette-Semmelmann,
    {\cite[Theorem 1.9]{MR1876863}})}
\label{thm:specialhol1}
Let $(M,g)$ be a connected K\"ahler-Einstein manifold with positive
Ricci curvature.  Let $L$ be the anti-canonical line bundle on $M$.
Then $(M,L)$ does not admit $(g,A)$ with $\Rgen_{(g,A)}>0$, but with
the metric $g$ and connection $A$ on $L$ with harmonic curvature
{\lp}necessarily a multiple of $g${\rp} $\Rgen_{(g,A)}=0$.
\end{theorem}  
\begin{proof}
  Note that $\alpha^c(M,L)$ for this choice of $L$ coincides with the
  Todd genus, which is the alternating sum of the Hodge numbers
  $h^{0,q}(M)$.  Since the Ricci curvature is positive, the $h^{0,q}$
  vanish for $q>0$ by the Kodaira vanishing theorem, the Bochner
  formula, and Serre duality, so the Todd genus is $1$ \cite[Corollary
    IV.11.12]{lawson89:_spin}.  Thus $M$ cannot admit positive
  generalized scalar curvature. On the other hand, since $M$ is K\"ahler-Einstein, the
  curvature $2$-form of $L$ (with its standard connection) coincides
  with the Ricci form, and the Ricci tensor is $\frac{R_g}{2}g$.  That
  implies that $2|\Omega|_{op,g}=R_g$, so that $\Rgen_g=0$.  (See
  \cite[Lemma IV.11.10]{lawson89:_spin} for a similar calculation.)
\end{proof}
\begin{remark}
  {While Theorem \ref{thm:specialhol1} and \cite[Theorem
      1.9]{MR1876863} both address the nonexistence of metrics with
    scalar curvature bounded below by curvature form norms, it is important
    to note that the specific norms used in each theorem are
    distinct.}
\end{remark}  

\section{The Index Difference Map}
\label{sec:inddiff}

We start with some basic definitions.
\begin{definition}
\label{def:inddiff}
Let $M$ be a closed spin manifold, $\dim M=n$, with a fixed psc metric
$g_0$.  The index difference homomorphism 
\begin{equation*}
\inddiff(M)\co \pi_q(\cR^+(M))\to KO_{q+n+1}
\end{equation*}
is defined by sending the homotopy class of $f \co (S^q,*) \to
(\cR^+(M),g_0)$ to the index obstruction to filling in $f$ to a map
$(D^{q+1},*)\to (\cR^+(M),g_0)$, or in other words the $KO_*$-valued
index of the Dirac operator on the spin manifold
$ N^{n+q+1}=M^n\times
  D^{q+1}\cup_{M^n\times S^q\times \{0\}}M^n\times S^q\times [0,\infty)$
  with Riemannian metric restricting to $f(x)+dt^2$ on $M\times \{x\}
  \times [0,\infty)$, for $x\in S^q$.
\end{definition}
The index above is well-defined because the scalar curvature of $N$ is
uniformly positive outside of a compact set, and thus the Dirac
operator on $N$ is Fredholm{{;} see
  \cite{MR3683115,MR3961330} for detailed description{s} of the
  index difference.} 

We can do something analogous in the spin$^c$ case.
\begin{definition}
Let $(M,L)$ be a spin$^c$ manifold with a fixed choice of basepoint
$(g_0,A_0)\in \cR^{\gen}(M,L)$. The \emph{spin$^c$ index difference map}
\begin{equation*}
\inddiff^c(M,L)\co \pi_q(\cR^{\gen}(M,L))\to KU_{q+n+1}
\end{equation*}
is defined just as in the spin case, by sending the homotopy class of
\begin{equation*}
f\co (S^q,*) \to (\cR^{\gen}(M,L),(g_0,A_0))
\end{equation*}
to the index obstruction to filling in
$f$ to a map $(D^{q+1},*)\to (\cR^{\gen}(M,L),(g_0,A_0))$,
by taking the usual $KU_*$-valued index of
the complex Dirac operator $D_{\tilde L}$ on the spin$^c$ manifold
$(N^{n+q+1},\tilde L):=(M^n\times D^{q+1}\cup_{M^n\times S^q\times
  \{0\}}M^n\times S^q\times [0,\infty), \tilde L)$, with the same
condition on a Riemannian metric as in the spin case, and with $\tilde
L = p^*L$, where $p$ is the natural projection $p: N^{n+q+1}\to M$.
\end{definition}
The following propositions show that the index difference map
is nontrivial (Theorem \ref{thm:psc-space}, (i)).
\begin{proposition}
\label{prop:manycomponents}
Let $(M,L)$ be a closed connected $n$-dimensional spin$^c$ manifold,
with $n=4k-1$, $k\ge 2$ and admitting generalized positive scalar
curvature. Then the {map
  \[
  \inddiff^c(M,L)\co\pi_0(\cR^{\gen}(M,L))\to \bZ
  \]
takes infinitely many nonzero
  values}, and in particular, $\cR^{\gen}(M,L)$ 
has infinitely many connected components.
\end{proposition}
\begin{proof}
  We choose a $2$-connected spin manifold $N^{4k}$ with $\widehat
  A(N)\ne 1$. If $k$ is even, this can be a product of copies of the
  ``Bott manifold,'' for which a construction is given in
  \cite{Bott-manifold}.  If $k\geq 2$ is odd, $\widehat A(N)$ is
  necessarily even and one can construct $N$ as follows.  Use an $E_8$
  plumbing \cite{plumbing} to get a parallelizable $4k$-manifold
  $Z^{4k}$ with boundary a homotopy sphere $\Sigma^{4k-1}$ (a
  generator of $bP_{4k}$). This homotopy sphere has finite order in
  the group $\Theta_{4k-1}$ of homotopy spheres, so we can take a
  boundary connected sum of finitely many copies of $Z^{4k}$ to get a
  parallelizable manifold $V^{4k}=Z^{4k}\#_{\partial}\cdots\#_{\partial}Z^{4k}$
  with boundary $S^{4k-1}$.  Glue in a disk:
  \begin{equation*}
  N^{4k}:= V^{4k}\cup_{S^{4k-1}} D^{4k}.
  \end{equation*}  
  Now $N^{4k}$ is our desired manifold: it is closed $4k$-manifold which
  is highly connected, hence spin, and has only one nonzero {Pontryagin
  class}, the top Pontryagin class $p_k$. Since
  the signature of $Z^{4k}$ is $8$, $N$ has nonzero signature and
  nonzero $\widehat A$-genus.

  Now we turn to our spin$^c$ manifold $(M,L)$. We choose $(g,h,A)\in
  \cR^{\gen}(M,L)$ and consider the spin$^c$ bordism $W=(M\times
     [0,1]) \# N$ from $M$ to itself, where $N$ is attached to the
     interior of the cylinder $M\times [0,1]$. The bordism $W$
   has natural spin$^c$ structure coming from $M$, since the spin$^c$
   structure on $M$ defines one on the cylinder $M\times [0,1]$ and
   then this spin$^c$ structure extends uniquely to $W$ since $N$ is
   $2$-connected.  We use a product metric $g+dt^2$ to extend the
   metric from $M\times \{0\}$ to $M\times [0,\varepsilon]$ for some
   $\varepsilon>0$.  By the proof of the bordism theorem \cite[Theorem
     4.3]{BJ} and the fact that the inclusion of either boundary copy
   of $M$ into $W$ is $2$-connected, we can push the triple $(g,h,A)$
   across the bordism $W$ to obtain $(g',h',A')\in \cR^{\gen}(M,L)$ at
   $M\times \{1\}$.  The value of the index difference is then
   obtained by computing the index of the Dirac operator on the
     manifold
\begin{equation*}
{W':=} M\times (-\infty,0]\cup_{M\times\{0\}}W\cup_{M\times\{1\}}M\times
     [1,\infty)
\end{equation*}
using $(g+dt^2,h,A)$ on $M\times (-\infty,0]$ and $(g'+dt^2,h',A')$ on
     $M\times [1,\infty)$.  This is just the sum of the index of Dirac
       on $M\times \bR$, which vanishes, with $\widehat
       A(N)$,\footnote{{\ This is because there is bordism, just a
         product with $[0,1]$ off a compact set, between our manifold
         {$W'$} and $(M\times \bR, g+dt^2)\sqcup N$, for which
         the index is obviously the sum {of indices of the
           disjoint manifolds}, and because of the bordism invariance
         of the index of the Dirac operator.}}  and so the index
       difference is nontrivial.  {We can now get infinitely many
         different values for the index difference by replacing $N$ by
         $\overbrace{N\#N\#\cdots\#N}^k$ in this construction
         to get an index difference of $k\widehat A(N)$ (for any
         positive $k$).}
\end{proof}

Before stating the next result, we need some notation.  Let $ c\co
KO_n\to KU_n $ denote the ``complexification map'' {\lp}which is zero if $n$ is
not divisible by $4$, an isomorphism if $n\equiv0 \pmod8$, and a
monomorphism with cokernel $\bZ/2$ if $n\equiv4 \pmod8${\rp}.  We also
need an observation: if $n$ is divisible by 4, there is always a closed
$n$-dimensional spin$^c$ manifold $(M,L)$ which is not spin with
$\alpha^c(M,L)=1$.  For example, we can take $M=\bC\bP^{n/2}$ and
$L=\cO\bigl(1+\frac{n}{2}\bigr)$.  Take such an $(M,L)$
and fix a pair $(g_M,A)$, which \emph{cannot} have
generalized scalar curvature $\Rgen_{(g,A)}$ positive since
the index of $D_L$ is nonzero.

Consider the spin$^c$ manifold $(S^k\times M,p^*L)$, where $k\geq 2$,
and let $p\co S^k\times M\to M$ be projection on the second factor. The
connection $A$ on $L$ gives a connection $p^*A$ on $p^*L$. Let
$\cR^+(S^k)$ be the space of psc metrics on $S^k$, with basepoint the
standard round metric.  Now we define a map
\begin{equation}\label{eq:sigma}
  \sigma\co \pi_q(\cR^+(S^k))\to \pi_q(\cR^{\gen}(S^k\times M))
\end{equation}  
as follows. Fix a map $f\co S^q\to \cR^+(S^k)$ representing a
homotopy class, and note that we can find a scale $t>0$
small enough so that the generalized scalar curvature
\begin{equation*}
  \Rgen_{(t^2g+g_M,p^*A)}=(t^{-2}R_{g}+ R_{g_M}) - 2|\Omega|_{op}
\end{equation*}    
is positive for all $g\in f(S^q)\subset \cR^+(S^k)$, the compact image
of our map $S^q\to \cR^+(S^k)$.  Then the formula $x\mapsto
(t^2f(x)+g_M,p^*A)$, $x\in S^q$, defines $\sigma([f])$, since it's
clear that up the homotopy class of this family of metrics only
depends on the homotopy class of $f$ and not on the parameter $t$.
Once $t$ is fixed, $\sigma$ is obtained by taking the product of the
family of metrics $t^2f$ (which represents the same homotopy class as
$f$) with the metric $g_M$ on $M$, and this operation of direct
product is compatible with addition of homotopy classes, so in this
way we get a homomorphism $\pi_q(\cR^+(S^k))\to
\pi_q(\cR^{\gen}(S^k\times M))$.  We obtain the following diagram:
\begin{equation}\label{eq:compatibility}
\begin{diagram}
\setlength{\dgARROWLENGTH}{1.2em}
       \node{\pi_q\cR^+(S^k)}
           \arrow[5]{e,t}{\inddiff(S^k)}
           \arrow{s,r}{\sigma}
       \node[5]{KO_{q+k+1}}
           \arrow{s,r}{\tilde c}
           \arrow[2]{e,t}{c}
           \node[2]{KU_{q+k+1}}
           \arrow{wsw,b}{\mathbb{B}}
           \\
       \node{\pi_q\cR^{\gen}(S^k\times M)}
           \arrow[5]{e,t}{\inddiff^c(S^k\times M,p^*L)}
       \node[5]{KU_{q+k+n+1}}
\end{diagram}
\end{equation}
where $\tilde c=\mathbb{B}\circ c$, and $\mathbb{B}$ is the Bott
isomorphism.
\begin{proposition}
  Let $(M,L)$ be as above and $q+k+1$ be a multiple of 4. Then the
  diagram {\rm (\ref{eq:compatibility})} commutes.
\end{proposition}
\begin{proof}
Suppose we are given a map $g\co (S^q,*)\to (\cR^+(S^k),g_0)$, where
$g_0$ is the standard metric.  By scaling $g$ we obtain a map
$\sigma(g)\co(S^q,*)\to (\cR^{\gen}(S^k\times M),g_0+g_M)$. Then
$\inddiff(g)$ is defined by extending $g$ from a family of metrics on
$S^q\times S^k$ to a family of metrics on $D^{q+1}\times S^k$ which
are product metrics on a collar neighborhood of the boundary, and for
each $x\in S^q$, extending via a cylinder metric $g(x)+ dt^2$ on
$S^k\times [0, \infty)$ to get a complete metric on
$\bR^{q+1}\times S^k$ with positive scalar curvature off a compact
set.  The complex Dirac operator for this metric is Fredholm, and has
index given by $c(\inddiff(g))$.  Now take the Riemannian product with
$(M, g_M)$ and equip this with the line bundle $p^*L$.  We can repeat
the same construction as before with the spin$^c$ Dirac operator, and
since this is the external product of the spin$^c$ Dirac on $M$ with
the Dirac on $D^{q+1}\times S^k$, the indices multiply {via
  the multiplicativity of the index as explained in
  \cite[\S5 and \S9]{MR0236950};}
here we use
that $\alpha^c(M,L)=1$.  So the conclusion follows.
\end{proof} 

Now we can prove (ii) from Theorem \ref{thm:psc-space}.
\begin{corollary}
\label{cor:manpiq}
There exists a simply connected spin$^c$ manifold $(M,L)$,
not spin, for which
infinitely many homotopy groups of $\cR^{\gen}(M,L)$ are non-zero.
\end{corollary}
\begin{proof}[Proof of Corollary]
  Simply take $M$ to be $\bC\bP^2\times S^7$, $L=p^*\cO(3)$,
  $p$ the projection of $\bC\bP^2\times S^7$ onto $\bC\bP^2$.
  If $x$ is the standard generator of $H^2(\bC\bP^2;\bZ)$, the $\widehat A$
  class is $\widehat {\mathbf A}(\bC\bP^2)=1-\frac{x^2}{8}$ and so
  $\ind D_{\cO(3)}$ is the coefficient of $x^2$ in
  $\left(1-\frac{x^2}{8}\right)(1+\frac{3x}{2}+\frac{9x^2}{8})$,
  which is $1$, and use the known
  nontriviality of the index difference map for $S^7$ (see for
  example \cite[Theorem A]{MR3681394}).
\end{proof}
\begin{remark}
  We stated in Conjecture \ref{BR-conj} that nontriviality of the
  topology of $\cR^{\gen}(M,L)$ is in fact quite a general
  phenomenon. So one shouldn't attach much significance to the special
  nature of this example.
\end{remark}


\bibliographystyle{amsplain}
\bibliography{SpincMetrics}

\providecommand{\bysame}{\leavevmode\hbox to3em{\hrulefill}\thinspace}
\providecommand{\MR}{\relax\ifhmode\unskip\space\fi MR }
\providecommand{\MRhref}[2]{%
  \href{http://www.ams.org/mathscinet-getitem?mr=#1}{#2}
}
\providecommand{\href}[2]{#2}
\begin{thebibliography}{10}

\bibitem{MR0236950}
M.~F. Atiyah and I.~M. Singer, \emph{The index of elliptic operators. {I}},
  Ann. of Math. (2) \textbf{87} (1968), 484--530. \MR{236950}

\bibitem{plumbing}
Manifold Atlas, \emph{Plumbing}, available at {\texttt
  {http://www.map.mpim-bonn.mpg.de/Plumbing}}.

\bibitem{MR1162671}
Christian B\"{a}r, \emph{Lower eigenvalue estimates for {D}irac operators},
  Math. Ann. \textbf{293} (1992), no.~1, 39--46. \MR{1162671}

\bibitem{MR3681394}
Boris Botvinnik, Johannes Ebert, and Oscar Randal-Williams, \emph{Infinite loop
  spaces and positive scalar curvature}, Invent. Math. \textbf{209} (2017),
  no.~3, 749--835. \MR{3681394}

\bibitem{BJ}
Boris Botvinnik and Jonathan Rosenberg, \emph{Positive scalar curvature on
  manifolds with fibered singularities}, J. Reine Angew. Math. \textbf{803}
  (2023), 103--136. \MR{4649186}

\bibitem{Bott-manifold}
Johannes Ebert, \emph{Construction of a {B}ott manifold}, MathOverflow
  {\#168296}, available at {\texttt
  {https://mathoverflow.net/questions/168296/construction-of-a-bott-manifold}}.

\bibitem{MR3683115}
\bysame, \emph{The two definitions of the index difference}, Trans. Amer. Math.
  Soc. \textbf{369} (2017), no.~10, 7469--7507. \MR{3683115}

\bibitem{MR3961330}
\bysame, \emph{Index theory in spaces of manifolds}, Math. Ann. \textbf{374}
  (2019), no.~1-2, 931--962. \MR{3961330}

\bibitem{MR3956897}
Johannes Ebert and Oscar Randal-Williams, \emph{Infinite loop spaces and
  positive scalar curvature in the presence of a fundamental group}, Geom.
  Topol. \textbf{23} (2019), no.~3, 1549--1610. \MR{3956897}

\bibitem{MR3936016}
Oliver Goertsches and Panagiotis Konstantis, \emph{Almost complex structures on
  connected sums of complex projective spaces}, Ann. K-Theory \textbf{4}
  (2019), no.~1, 139--149. \MR{3936016}

\bibitem{MR1876863}
S.~Goette and U.~Semmelmann, \emph{{${\rm Spin}^c$} structures and scalar
  curvature estimates}, Ann. Global Anal. Geom. \textbf{20} (2001), no.~4,
  301--324. \MR{1876863}

\bibitem{MR508087}
Akio Hattori, \emph{{${\rm Spin}^{c}$}-structures and {$S^{1}$}-actions},
  Invent. Math. \textbf{48} (1978), no.~1, 7--31. \MR{508087}

\bibitem{MR1705917}
Marc Herzlich and Andrei Moroianu, \emph{Generalized {K}illing spinors and
  conformal eigenvalue estimates for {${\rm Spin}^c$} manifolds}, Ann. Global
  Anal. Geom. \textbf{17} (1999), no.~4, 341--370. \MR{1705917}

\bibitem{MR834486}
Oussama Hijazi, \emph{A conformal lower bound for the smallest eigenvalue of
  the {D}irac operator and {K}illing spinors}, Comm. Math. Phys. \textbf{104}
  (1986), no.~1, 151--162. \MR{834486}

\bibitem{MR0375154}
Jerry~L. Kazdan and F.~W. Warner, \emph{A direct approach to the determination
  of {G}aussian and scalar curvature functions}, Invent. Math. \textbf{28}
  (1975), 227--230. \MR{375154}

\bibitem{MR0375153}
\bysame, \emph{Existence and conformal deformation of metrics with prescribed
  {G}aussian and scalar curvatures}, Ann. of Math. (2) \textbf{101} (1975),
  317--331. \MR{375153}

\bibitem{MR0394505}
\bysame, \emph{Prescribing curvatures}, Differential geometry ({P}roc.
  {S}ympos. {P}ure {M}ath., {V}ol. {XXVII}, {P}art 2, {S}tanford {U}niv.,
  {S}tanford, {C}alif., 1973), Proc. Sympos. Pure Math., vol. Vol. XXVII, Part
  2, Amer. Math. Soc., Providence, RI, 1975, pp.~309--319. \MR{394505}

\bibitem{MR365409}
\bysame, \emph{Scalar curvature and conformal deformation of {R}iemannian
  structure}, J. Differential Geometry \textbf{10} (1975), 113--134.
  \MR{365409}

\bibitem{lawson89:_spin}
Blaine Lawson and Marie-Lousie Michelsohn, \emph{Spin geometry}, Princeton
  University Press, 1989.

\bibitem{MR100267}
John Milnor, \emph{On spaces having the homotopy type of a {${\rm
  CW}$}-complex}, Trans. Amer. Math. Soc. \textbf{90} (1959), 272--280.
  \MR{100267}

\bibitem{MR1463835}
Andrei Moroianu, \emph{Parallel and {K}illing spinors on {${\rm Spin}^c$}
  manifolds}, Comm. Math. Phys. \textbf{187} (1997), no.~2, 417--427.
  \MR{1463835}

\bibitem{MR2661160}
Roger Nakad, \emph{Lower bounds for the eigenvalues of the {D}irac operator on
  {${\rm Spin}^c$} manifolds}, J. Geom. Phys. \textbf{60} (2010), no.~10,
  1634--1642. \MR{2661160}

\end{thebibliography}

\end{document}